\documentclass{article}

\usepackage{amsmath}
\usepackage{amssymb}
\usepackage{amsthm}
\usepackage[english]{babel}

\newenvironment{smatrix}{\left(\begin{smallmatrix}}{\end{smallmatrix}\right)}

\def\a{\alpha} \def\b{\beta}  \def\f{\varphi}
 \def\s{\sigma}  
  
\def\minus{\smallsetminus}
\def\({\left(} \def\){\right)}
\def\<{\langle} \def\>{\rangle}

 \def\sp{\mathrm{span}}

\newtheorem{thm}{Theorem}
\newtheorem{lma}[thm]{Lemma}
\newtheorem{prop}[thm]{Proposition}
\newtheorem{cor}[thm]{Corollary}

\renewcommand{\le}{\leqslant}

\title{Vector product algebras}
\author{Erik Darp\"{o}  \\ {\small \textit{Matematiska Institutionen, Uppsala
  Universitet,}}\\{\small\textit{Box 480, S-75106 Uppsala,
  Sweden.}}\\{\small\texttt{erik.darpo@math.uu.se}}}

\begin{document}
\selectlanguage{english}
\date{}
\maketitle 

\begin{abstract}
Vector products can be defined on spaces of dimensions 0, 1, 3 and 7 only, and their
isomorphism types are determined entirely by their adherent symmetric bilinear forms.
We present a short and elementary proof for this classical result.
\end{abstract}

\textit{Keywords: Vector product algebra, composition algebra. }

\section{Introduction}

Throughout this article, $k$ denotes a field of characteristic different from $2$.

A \emph{vector product algebra} over $k$ is a vector space $V$ over $k$ equipped with
an anti-symmetric bilinear map
\begin{equation*}
  V\times V\to V,\; (u,v)\mapsto uv
\end{equation*}
and a non-degenerate symmetric bilinear form $\<\,\>:V\times V\to k$
such that
\begin{enumerate}%\renewcommand\labelenumi{(\roman{enumi})}
\item\label{d1} $\<uv,w\>=\<u,vw\>$ and
\item\label{d2} $\<uv,uv\>=\<u,u\>\<v,v\> -\<u,v\>^2$
\end{enumerate}
for all $u,v,w\in V$.
A morphism between vector product algebras $V$ and $W$ is an orthogonal map $\f:V\to W$
such that $\f(uv)=\f(u)\f(v)$ for all $u,v\in V$.
Two vector product algebras $V$ and $W$ are isomorphic if there exists a bijective
morphism $\f:V\to W$.

A related concept is the one of a \emph{composition algebra}.
A composition algebra is a non-zero vector space $A$ together with a bilinear 
multiplication map $A\times A\to A,\;(x,y)\mapsto xy$ and a non-degenerate symmetric
bilinear form $\<\,\>$ such that $\<xy,xy\>=\<x,x\>\<y,y\>$ for all
$x,y\in A$. 
A morphism between composition algebras $A$ and $B$ is an orthogonal
map $\f:A\to B$ satisfying $\f(xy)=\f(x)\f(y)$ for all $x,y\in A$.

Given a vector product algebra $V$ over $k$, the algebra $\mathcal{H}(V)=k\times V$ with
multiplication
\begin{equation*}
  \begin{pmatrix}\a\\v\end{pmatrix}\begin{pmatrix}\b\\w\end{pmatrix}=
      \begin{pmatrix}\a\b-\<v,w\>\\\a w + \b v + vw\end{pmatrix}
\end{equation*}
and bilinear form
\begin{math}
  \left\<\begin{smatrix}\a\\v\end{smatrix},\begin{smatrix}\b\\w\end{smatrix}\right\>=
    \a\b+\<v,w\>
\end{math}
is a composition algebra with identity element 
$\begin{smatrix}1\\0 \end{smatrix} \in k\times V$. Conversely, every composition 
algebra with identity element is isomorphic to $\mathcal{H}(V)$ for some vector product
algebra $V$.
On the other hand, taking $V$ to be the orthogonal complement (with respect to $\<\,\>$) of the identity element in a
unital composition algebra $A$, the product 
\begin{equation} \label{vpmult}
  u\times v = \frac{1}{2}(uv-vu) \in V \qquad \mbox{for} \quad u,v\in V
\end{equation}
makes $V$ a vector product algebra, with the bilinear form induced from $A$.
These construction actually determine an equivalence between the respective categories of
vector product algebras and composition algebras with identity element
(c.f.~\cite{vectorcross}).

Unital composition algebras have been extensively studied (some examples are
\cite{albert41,dickson19,jac58,kaplansky}). They occur only in dimensions 
1, 2, 4 and 8 (so in particular they are all finite-dimensional), and can be constructed from
the ground field via application of the Cayley-Dickson process. Two unital composition
algebras are isomorphic if and only if their respective bilinear forms are equivalent. See
e.g. \cite{jac58} for details.
The above immediately yields the following statement for vector product algebras. 
\begin{thm}\label{main}
  \begin{enumerate}
  \item The dimension of any vector product algebra is either 0, 1, 3 or 7. \label{main1}
  \item Two vector product algebras are isomorphic if and only if their respective
  bilinear forms are equivalent. \label{main2} 
  \end{enumerate}
\end{thm}

Theorem~\ref{main} first appears in the article \cite{vectorcross} by Brown and Gray
(1967). The authors use the equivalence between vector product algebras and unital
composition algebras (which they prove) to reach the result.
Rost~\cite{rost} has studied vector product algebras directly, without using the relation
with composition algebras. His result tells, that if $d$ is the dimension of a vector
product algebra, then the identity $d(d-1)(d-3)(d-7)=0$ holds in the ground field. 
An alternative proof for this identity (which is weaker than Theorem~\ref{main}) has also
been given by Meyberg~\cite{meyberg}.

Over the real ground field, with $\<\,\>$ being a scalar product, vector products were
first considered, and classified, by Eckmann \cite{eckmann} in 1942, using topological
methods. Other treatments are found in for example \cite{massey} and
\cite{composition}. The technique of the present article is used in \cite{svp} for a
comprehensive proof of the classification theorem in this special case.

In this article, we present a direct proof of Theorem~\ref{main}, avoiding the detour via
composition algebras. In spite of its simplicity, the proof idea seems to have gone
unnoticed in the literature so far. 
The basic line of reason is the following. If $W\subsetneq V$ is a subalgebra of a vector
product algebra $V$, then by adjoining an additional element, say $e$, to $W$, one gets a
subalgebra $\<W\cup\{e\}\>\subset V$ of dimension $2\dim W+1$. Hence, starting with the
trivial subalgebra $V_0=0\subset V$, one builds inductively a chain 
$V_0\subset V_1\subset\dots$ of subalgebras of $V$ with $\dim V_i=2^i-1$. 
The process breaks down in the fourth step, where the algebra $V_4$ fails to satisfy the
axioms of a vector product algebra. Therefrom one concludes, that $V=V_i$ for $i\le 3$,
and thus $\dim V\in\{0,1,3,7\}$.
Also the isomorphism statement follows readily from this argument.

The following notation and definitions are used. 
The quadratic form corresponding to a symmetric bilinear form $\<\,\>$ is denoted $N$,
i.e., $N(v)=\<v,v\>$. Conversely, 
\begin{equation} \label{q}
  \<u,v\>=\frac{1}{2}(N(u+v)-N(u)-N(v))
\end{equation}
holds.
Given spaces $V$ and $W$ with bilinear forms $\<\,\>_V$ and $\<\,\>_W$ respectively,
denote by $V\perp W$ their orthogonal direct sum, that is $V\oplus W$ with bilinear form
$\<u+v,w+x\>=\<u,w\>_V+\<v,x\>_W$ for $u,w\in V,\;v,x\in W$.
A non-zero element $v\in V$ is called
\emph{isotropic} if $N(v)=0$ and anisotropic otherwise. 
If $U\subset V$ is a subspace then $U^\perp$ denotes its orthogonal complement in $V$ with
respect to $\<\,\>$.
%$U^\perp=\{v\in V\mid \forall u\in U: \<u,v\>=0 \}$.

\section{Proof of Theorem~\ref{main}} \label{reduktion}

In this section, $V$ and $W$ always denote vector product algebras over the ground field
$k$.
The following lemma lists some properties of, and provides means to control the
multiplication in the vector product algebra $V$. 

\begin{lma}\label{vm} For all $u,v,w\in V$, the following identities hold.
  \begin{enumerate}
  \item $\<u,uv\>=0$ \label{ii}
  \item $\<uv,uw\>=N(u)\<v,w\>-\<u,v\>\<u,w\>$ \label{iii}
  \item $u(vu)=N(u)v-\<u,v\>u$ \label{iv}
  \end{enumerate}
  If $u,v,w\in V$ are pairwise orthogonal with respect to $\<\,\>$, then
  \begin{enumerate}\setcounter{enumi}{3}
  \item $u(vw)=-(uv)w=(vu)w$\label{v}
  \end{enumerate}
\end{lma}

\begin{proof} The identity \ref{d1} in the definition gives
$\<u,uv\>=\<u^2,v\>=\<0,v\>=0$.
For \ref{iii}, we use the second identity in the definition together with
Equation~(\ref{q}) and compute 
\begin{multline*}
  2\<uv,uw\>=N(uv+uw)-N(uv)-N(uw)\\
  =N(u)N(v+w)-\<u,v+w\>^2-(N(u)N(v)-\<u,v\>^2)-(N(u)N(w)-\<u,w\>^2)\\
  =N(u)(N(v+w)-N(v)-N(w))-\<u,v+w\>^2+\<u,v\>^2+\<u,w\>^2\\
  =2N(u)\<v,w\>-2\<u,v\>\<u,w\>.
\end{multline*}

If $x\in V$ is an arbitrary vector, then
\begin{multline*}
  \<x,u(vu)\>=\<xu,vu\>=\<ux,uv\>=N(u)\<x,v\>-\<u,x\>\<u,v\>\\
  =\<x,N(u)v-\<u,v\>u\>
\end{multline*}
which implies
\begin{equation*}
  \<x,u(uv)-(N(u)v-\<u,v\>u)\>=0.
\end{equation*}
Since $\<\,\>$ is non-degenerate, it follows that $u(vu)=N(u)v-\<u,v\>u$.

Finally, using clause \ref{iv} and the orthogonality of $u,v,w$ we get
\begin{multline*}
  N(u+w)v=(u+w)(v(u+w))=u(vu)+u(vw)+w(vu)+w(vw)\\=N(u)v+u(vw)+w(vu)+N(w)v
\end{multline*}
and so
\begin{equation*}
  u(vw)=-w(vu)+(N(u+w)-N(u)v-N(w))v=(vu)w+2\<u,w\>=(vu)w,
\end{equation*}
which concludes the proof of the lemma.
\end{proof}

Our basic tool of investigation will be \emph{multiplicatively independent sets}.
We say that a finite subset $E\subset V$ is multiplicatively independent if any $e\in E$ is
anisotropic and orthogonal to the subalgebra $\<E_e\>\subset V$ generated by
$E_e=E\minus\{e\}$. Clearly, every subset of a multiplicatively independent set is also
multiplicatively independent.

Let $A=\{a_1,\ldots,a_n\}\subset V$ be a non-empty finite ordered set. We write
\begin{equation*}
  \Pi(A)=
  \left\{
  \begin{array}{ll}
    a_1 &\mbox{if } n=1 \\ 
    (\Pi(A\!\minus\!\{a_n\}))a_n  &\mbox{if } n>1
  \end{array}
  \right.,
\end{equation*}
that is, $\Pi(A)$ is the ordered product of the elements in $A$ with brackets distributed
as ``far to the left'' as possible. It may be remarked that if $A$ is multiplicatively
independent, which will be our case of concern, the order of the elements, as well as
the distribution of brackets, is essentially unimportant. From Lemma~\ref{vm} follows,
that alterations of these data will at most change the sign of the product.

\begin{prop}\label{E}
  Let $E\subset V$ be a multiplicatively independent set and let $I$ be the set of
  non-empty subsets of $E$, each subset equipped with some fixed order. Then
  the following hold.
  \begin{enumerate}
  \item The set $\{\Pi(A)\}_{A\in I}$ is an orthogonal basis for the vector space
  $\<E\>$. \label{bas}
  \item If $e\in E$, then $\<E\>=\<E_e\>\perp\<e\>\perp\<E_e\>e$. \label{sum}
  \item If $f\in V$ is anisotropic and ortho\-gonal to $\<E\>$, then $E\cup\{f\}$ is a
    multiplicatively independent set. \label{utv}
  \end{enumerate}
\end{prop}

\begin{proof} 
We write $\Sigma=\<E_e\>+\<e\>+\<E_e\>e$. Clearly,
$\Sigma\subset\<E\>$. On the other hand, from Lemma~\ref{vm} follows that $\Sigma$ is
closed under multiplication. As $E\subset\Sigma$, we get
$\<E\>\subset\<\Sigma\>=\Sigma$. If $m$ is the number of elements in $E$, it follows by
induction that $\dim\<E\>=\dim\Sigma\le2^m-1$.

Now suppose that $A,B\in I$, with $A=\{a_i\}_{i=1}^l$ and $a_j\not\in B$ for some 
$j\le l$. Since the order on $A$ only affects the sign of $\Pi(A)$, for our purpose we may
assume that $j=l$. Hence
\begin{multline*}
  \<\Pi(A),\Pi(B)\>=
  \<\Pi(A\minus\{a_l\})a_l,\Pi(B)\>=-\<a_l,\Pi(A\minus\{a_l\})\Pi(B)\>=0,
%%   \<(\ldots((a_1a_2)a_3)\ldots)a_j,a_{j+1}(a_{j+2}(\ldots(a_l\Pi(B))\ldots))\>\\
%%   -\<a_j,(\ldots((a_1a_2)a_3)\ldots)(a_{j+1}(a_{j+2}(\ldots(a_l\Pi(B))\ldots)))\>=0.
\end{multline*}
where the last equality follows from that $\Pi(A\minus\{a_l\})\Pi(B)$ is an
element in $\<E_{a_l}\>$, and hence orthogonal to $a_l$.
In addition, for $A=\{a_1,\ldots,a_l\}\in I$ we have
$N(\Pi(A))=\prod_{j=1}^{l}N(a_j)\neq0$. Since the vectors $\Pi(A),\;A\in I$ are
anisotropic and pairwise orthogonal, they are linearly independent, and
$\dim(\sp\{\Pi(A)\}_{A\in I})=2^m-1$. Hence
\begin{equation*}
  \sp\{\Pi(A)\}_{A\in I}=\<E\>=\<E_e\>\perp\<e\>\perp\<E_e\>e
\end{equation*}
and $\{\Pi(A)\}_{A\in I}$ is a basis for $\<E\>$.

For (\ref{utv}), we need to show that any $e\in E$ is orthogonal to $\<E_e\cup\{f\}\>$. By
induction, $E_e\cup\{f\}$ is multiplicatively independent, and thus
 $\<E_e\cup\{f\}\>=\<E_e\>\perp\<f\>\perp\<E_e\>f$. But 
$e\in \<E_e\>^\perp \cap \<f\>^\perp \cap (\<E_e\>f)^\perp$, which implies
$e\in\<E_e\cup\{f\}\>^\perp$.
\end{proof}

As already stated in the proof, Proposition~\ref{E} implies that if $E$ is a
multiplicatively independent set with $m$ elements, then $\dim\<E\>=2^m-1$.
Moreover, if $V$ is finite-dimensional, from Proposition~\ref{E}:\ref{utv} it follows that
$V=\<E\>$ for some multiplicatively independent set $E\subset V$. Such a set we call a
\emph{multiplicative base} for $V$. The first statement in the proposition implies the
following result for morphisms of vector product algebras.

\begin{cor} \label{morfism}
  \begin{enumerate}
  \item Let $E$ be a multiplicative base for $V$ and $\sigma:E\to W$ a map such that $\sigma(E)$
    is multiplicatively independent and $N(\sigma(e))=N(e)$ for all $e\in E$. Then $\sigma$
    extends uniquely to a morphism $V\to W$. \label{m1}
  \item The vector product algebras $V$ and $W$ are isomorphic if and only if there exist
    multiplicative bases $E\subset V$, and $F\subset W$ and a bijection $\s:E\to F$ such
    that $N(\sigma(e))=N(e)$ for all $e\in E$. \label{m2}
  \end{enumerate}
\end{cor}

We are now ready to conclude the proof of Theorem~\ref{main}.
Suppose that $E=\{u,v,w,z\}\subset V$ is a multiplicatively independent set. Evaluating
$u\(v(wz)\)$ we get 
\begin{align*}
  u(v(wz))&=u((wv)z)=((wv)u)z=-((vw)u)z\\
  \intertext{but also}
  u(v(wz))&=(vu)(wz)=(w(vu))z=((vw)u)z.
\end{align*}
This means that $u(v(wz))=0$, which contradicts the multiplicative independence of
$E$. Hence any multiplicatively independent set contains at most 3 elements. This rules
out the existence of infinite-dimensional vector product algebras, since in such algebras
we would be able to find multiplicatively independent sets of any finite cardinality. It
further tells that every vector product algebra has dimension 0, 1, 3 or 7.

Suppose $V$ and $W$ are vector product algebras with equivalent bilinear forms. Further
suppose that $E\subset V$ and $F\subset W$ are multiplicatively independent sets and that
$\s:E\to F$ is a bijection preserving the quadratic form $N$. By Corollary~\ref{morfism},
$\<E\>$ and $\<F\>$ are isomorphic, and thus in particular isometric.
Witt's theorem for quadratic forms now implies that $\<E\>^\perp$ and $\<F\>^\perp$ also
are isometric. Hence there exist $e\in\<E\>^\perp,\;f\in\<F\>^\perp$
such that $N(e)=N(f)\neq0$. Proposition~\ref{E} implies, that $E\cup\{e\}$ and
$F\cup\{f\}$ are multiplicatively independent, and
$\<E\cup\{e\}\>=\<E\>\perp\<e\>\perp\<E\>e$. We extend $\s$ to a bijection $E\cup\{e\}\to
F\cup\{f\}$ by setting $\s(e)=f$. Proceeding inductively, we get a bijection between
multiplicative bases of $V$ and $W$, preserving $N$. By Corollary~\ref{morfism}, the
algebras $V$ and $W$ then are isomorphic.

\section{Acknowledgements}

The author is indebted to Jos\'e Antonio Cuenca Mira, who suggested to generalise the note
\cite{svp} to general ground fields, and gave valuable advice for this generalisation.

\bibliographystyle{plain}
\bibliography{../litt.bib}

\end{document}